\documentclass[12pt, hyperref]{article}
\usepackage{amsfonts}
\usepackage{amssymb,latexsym,amsmath, amsthm}
\usepackage{hyperref}
\usepackage{amscd}
\usepackage{graphicx} 
\usepackage{graphicx,epstopdf,subfigure}
\usepackage{epsfig}
\usepackage{subfigure}
\usepackage[all]{xy}
\usepackage{xcolor}
\hypersetup{linkcolor=blue, pdfstartview=FitH} \hypersetup{colorlinks, linkcolor=blue,
pdfstartview=FitH}
\newtheorem{corollary}{Corollary}
\newtheorem{lemma}[corollary]{Lemma}
\newtheorem{definition}[corollary]{Definition}

\newtheorem{theorem}[corollary]{Theorem}

\begin{document}

\title{\textbf{\large{Three Remarks on Carleson Measures for Dirichlet Space}}}

\author{\small{GUOZHENG CHENG,
XIANG FANG, ZIPENG WANG, AND JIAYANG YU}}
\date{}
\maketitle
\def\cc{\mathbb{C}}
\def\zz{\mathbb{Z}}
\def\nn{\mathbb{N}}
\def\rr{\mathbb{R}}
\def\qq{\mathbb{Q}}
\def\dd{\mathbb{D}}
\def\tt{\mathbb{T}}
\def\bb{\mathbb{B}}
\def\ff{\mathbb{F}}

\def\A{\mathcal{A}}
\def\B{\mathcal{B}}
\def\D{\mathcal{D}}
\def\L{\mathcal{L}}
\def\M{\mathcal{M}}
\def\Mperp{\mathcal{M}^\perp}
\def\N{\mathcal{N}}
\def\K{\mathcal{K}}
\def\F{\mathcal{F}}
\def\R{\mathcal{R}}
\def\s{\mathcal{S}}
\def\p{\mathcal{P}}
\def\P{\mathcal{P}}
\def\T{\mathcal{T}}
\def\O{\mathcal{O}}
\def\Z{\mathcal{Z}}
\def\hh{\mathbb{H}}

\def\cod{\text{cod}}
\def\fd{\text{fd}}
\def\ker{\text{ker}}
\def\ran{\text{ran}}
\def\mp{\text{mp}}
\def\mp{\text{mp}}
\def\ee{\mathbb{E}}
\def\re{\text{Re}}
\def\im{\text{Im}}

\def\al{\alpha}
\def\la{\lambda}
\def\ep{\epsilon}
\def\sig{\sigma}
\def\Sig{\Sigma}
\def\cd{\mathbb{C}^d}
\def\bm{\mathcal{M}}
\def\bn{\mathcal{N}}
\def\hN{H^2\otimes \mathbb{C}^N}
\def\ba{\mathcal{A}}
\def\hm{H^2\otimes \mathbb{C}^m}
\def\mb{\mathcal{M}^{\perp}}
\def\pr{{\mathbb C}[z_1,\cdots,z_n]}
\def\nb{\mathcal{N}^{\perp}}
\def\hrd{H^2(\mathbb{D}^n)}
\def\be{\mathcal{E}}
\def\po{{\mathbb C}[z,\,w]}
\def\hr{H^2({\mathbb D}^2)}
\def\bigpa#1{\biggl( #1 \biggr)}
\def\bigbracket#1{\biggl[ #1 \biggr]}
\def\bigbrace#1{\biggl\lbrace #1 \biggr\rbrace}

\def\papa#1#2{\frac{\partial #1}{\partial #2}}
\def\dbar{\bar{\partial}}

\def\oneover#1{\frac{1}{#1}}

\def\meihua{\bigskip \noindent $\clubsuit \ $}
\def\blue#1{\textcolor[rgb]{0.00,0.00,1.00}{#1}}
\def\red#1{\textcolor[rgb]{1.00,0.00,0.00}{#1}}

\def\norm#1{||#1||}
\def\inner#1#2{\langle #1, \ #2 \rangle}

\def\bigno{\bigskip \noindent}
\def\medno{\medskip \noindent}
\def\smallno{\smallskip \noindent}
\def\bignobf#1{\bigskip \noindent \textbf{#1}}
\def\mednobf#1{\medskip \noindent \textbf{#1}}
\def\smallnobf#1{\smallskip \noindent \textbf{#1}}
\def\nobf#1{\noindent \textbf{#1}}
\def\nobfblue#1{\noindent \textbf{\textcolor[rgb]{0.00,0.00,1.00}{#1}}}

\def\vector#1#2{\begin{pmatrix}  #1  \\  #2 \end{pmatrix}}

\def\cfh{\mathfrak{CF}(H)}

\begin{abstract}

  In this paper we prove that all doubling measures on the unit disk $\dd$ are Carleson measures for the standard  Dirichlet space $\mathcal{D}$. The proof has three new ingredients. The first one is a characterization of Carleson measures which holds true for general  reproducing kernel Hilbert spaces.
  The second one is another new equivalent condition for Carleson measures, which holds true only for the standard Dirichlet space. 
  The third one is an application of dyadic method to our setting.  
\end{abstract}

\footnote{
 \textbf{Keywords:}  Carleson measure; doubling measure; Dirichlet space.

~~~\textbf{Mathematics Subject Classification [MSC]:} 47B34, 47G10.}



\section{Introduction and the Main Result }
\noindent In this paper we study  Carleson measures on the standard Dirichlet space $\D$, which consists of analytic functions over the unit disk $\dd \subset \cc$ under the norm
$$\|f\|^2_\D=|f(0)|^2+\frac{1}{\pi}\int_\dd |f'(z)|^2 dA(z)<\infty.$$
 A positive measure $\mu$ on $\dd$ is called a $\D$-Carleson measure if
there exists a constant $C$ such that
$$
\int_{\dd}|f|^2d\mu\leq C\|f\|_{\D}^2,\ \ \ f\in \D.
$$
This is certainly a well studied topic in operator theory with a  long history.
Carleson measures were originally  introduced by L. Carleson in his 1962 solution of the corona problem, and have become one of the most cherished tools in analysis.  They have proved to be powerful in a great variety of problems.  In case of  the Hardy space or the Bergman space on the unit disk,  Carleson measures are characterized by beautiful geometric conditions  \cite{Zhu}. In contrast,   $\D$-Carleson measures remain somehow elusive after several decades of study and their characterization is   capacitary  \cite{Dirichlet book}, \cite{Ste1980}, \cite{Wu1999}, hence hard to check. This makes the search for sufficient conditions meaningful.
 In particular, the recent work of El-Fallah-Kellay-Mashreghi-Ransford \cite{EKMR2014}   contains an elegant sufficient condition for $\D$-Carleson measures, which  is especially effective if armed with  growth estimates near the boundary. This is  a recurrent theme for $\D$-Carleson measures in the past, both for characterizations and for sufficient conditions.
The purpose of this paper is to show that a large class of commonly used measures are all $\D$-Carleson.

\bigno Nowadays there exists an extensive literature on generalized Dirichlet spaces with generalized Carleson measures. Such aspects will not be considered in this paper. But we do  keep this group of readers in mind and make things convenient for them when it is at not much extra cost. (See the remark before Lemma \ref{L:twp reverse doubling}.)

\bigno
 A measure $\mu$ on the unit disk $\dd$ is called a doubling measure if there exists a constant $C$ such that
\begin{equation*}
\mu(B(z,2r))\leq C \mu(B(z,r))
\end{equation*}
for all $z\in\dd$ and $r>0$, where $B(z,r)=\{w\in\dd:|z-w|<r\}$. Doubling measures arise naturally in various ways in analysis and geometry and there is an industry building on them \cite{stein}. We restrict our attention to those  which are absolutely continuous with respect to the Lebesgue measure and we call them weights.

\begin{theorem}\label{T:main0405}If $\sigma$ is a finite reverse doubling measure that is absolutely continuous with respect to the Lebesgue measure on $\dd$, then it is a $\mathcal{D}$-Carleson measure.
\end{theorem}

\noindent Then \cite{FW2014} (Lemma 6) implies the claimed result for doubling measures. For the definition of reverse doubling, see  Definition \ref{D:R}. Our arguments can actually prove that  a  larger class of measures are $\D$-Carleson. But stating the result in its full strength will involve contrive technicality, hence undesirable, although our proof might attract some  to probe. Our approach to $\D$-Carleson measures is indeed quite different from the past.


\section{The Proof}
First we give a general characterization of Carleson measures on a reproducing kernel Hilbert space in terms of the reproducing kernel. This is a small but necessary extension of a  result of \cite{ARS}.

\begin{lemma}\label{T:Carleson,real kernel,kernel}
Let $\mu$ be a positive measure on $\dd.$ Suppose that $H$ is a Hilbert space of analytic functions over $\dd$ with a {reproducing kernel $k(z, w)$, such that for any $z\in\dd$, the function $k(z,\cdot)$ is continuous on $\overline{\dd}$.}  Let
\begin{equation*}\label{E:TKmu}
T_{k,\mu}f(z)=\int_{\dd}f(w)k(z,w)d\mu(w),\ \ \ f\in L^2(\dd, \mu),
\end{equation*}
and
\begin{equation*}\label{E:TRKmu}
T_{Re(k),\mu}f(z)=\int_{\dd}f(w)Re(k)(z,w)d\mu(w),\ \ \ f\in L^2(\dd, \mu),
\end{equation*}
 where $Re(k)$ denotes  the real part of $k.$
Then the following conditions are equivalent:
\begin{itemize}
\item[(a)] $T_{Re(k),\mu}: L^2(\dd, \mu) \to L^2(\dd, \mu)$ is bounded.
\item[(b)] $T_{k,\mu}: L^2(\dd, \mu) \to L^2(\dd, \mu)$ is bounded.
\item[(c)] $\mu$ is an $H$-Carleson measure. That is, there exists  a constant $c$ such that
$$
\|f\|_{L^2(\dd, \mu)}\leq c\|f\|_{H} \quad \text{for all}  \quad f\in H.$$
\end{itemize}
Moreover, in the above cases, we have $$\|T_{Re(k),\mu}\|_{L^2(\dd,\mu)\to L^2(\dd,\mu)}\leq\|T_{k,\mu}\|_{L^2(\dd,\mu)\to L^2(\dd,\mu)}\leq 2 \|T_{Re(k),\mu}\|_{L^2(\dd,\mu)\to L^2(\dd,\mu)}.$$
\end{lemma}

\begin{proof}
The equivalence between (a) and (c) was proved in \cite{ARS}. The equivalence between $(a)$ and  $(b)$ follows from the three lemmas below.

\begin{lemma}\label{L:bounded on real enough}
Let $\mu$ be a positive measure over $\dd.$
Suppose that $T$ is a linear operator (not necessarily bounded)  from $L^2(\dd,\mu)$ to some Hilbert space $H.$ If  there exists a constant $c$ such that
$$
\|Tf\|_{H}\leq c\|f\|_{L^2(\dd, \mu)}
$$
for all real-valued functions $f$ in $L^2(\dd, \mu),$ then
$T$ is bounded. Moreover,  $\|T\|_{L^2(\dd,\mu)\to H} \le \sqrt{2} c.$
\end{lemma}
\begin{proof} The proof is elementary and we include it for completeness. For $f\in L^2(\dd, d\mu),$ decompose it into real and imaginary parts as $f=f_1+if_2.$
Then
\begin{eqnarray*}
\big|\langle Tf, Tf\rangle_{H}\big|
&\leq& \big|\langle Tf_1, Tf_1\rangle_{H}\big|+ \big|\langle Tf_2, Tf_2\rangle_{H}\big|+ 2\big|\langle Tf_1, Tf_2\rangle_{H}\big|\\
&\leq& c^2\|f_1\|_{L^2(\dd, \mu)}^2+c^2\|f_2\|_{L^2(\dd, \mu)}^2+2c^2\|f_1\|_{L^2(\dd, \mu)}\|f_2\|_{L^2(\dd, \mu)}\\
&\leq&2c^2(\|f_1\|_{L^2(\dd,\mu)}^2+\|f_2\|_{L^2(\dd,\mu)}^2)\\
&\leq & 2c^2\|f\|_{L^2(\dd, \mu)}^2.
\end{eqnarray*}
\end{proof}

\begin{lemma}\label{L:T_K positive, T_K and T_RK}
Let $\mu$ be a positive measure over $\dd$ and $k:\dd\times\dd\to \cc$ be a measurable function such that for any $z\in\dd$ the function $k(z,\cdot)$ is continuous over $\overline{\dd}$.
If $T_{k,\mu}$ is positive, not necessarily bounded, on $L^2(\dd,\mu)$, then $T_{k,\mu}$ is bounded on $L^2(\dd,\mu)$
if and only if  $T_{Re(k),\mu}$ is bounded.
When $T_{Re(k),\mu}$ is bounded, we have
$$\|T_{Re(k),\mu}\|_{L^2(\dd,\mu)\to L^2(\dd,\mu)}\leq \|T_{k,\mu}\|_{L^2(\dd,\mu)\to L^2(\dd,\mu)}\le {2} \|T_{Re(k),\mu}\|_{L^2(\dd,\mu)\to L^2(\dd,\mu)}.$$

\end{lemma}
\begin{proof}Assume that $T_{k,\mu}$ is bounded on $L^2(\dd,\mu)$.
Let $$f\in L^2(\dd,\mu)$$ and $$g\in L^2(\dd,\mu),$$ we have
\begin{equation*}\label{E: real to all}
\langle T_{Re(k),\mu}f,g\rangle_{L^2(\dd,\mu)}=\frac{1}{2}\langle T_{k,\mu}f,g\rangle_{L^2(\dd,\mu)}+\frac{1}{2}\overline{\langle T_{k,\mu}\bar{f},\bar{g}\rangle}_{L^2(\dd,\mu)}.
\end{equation*}
Thus $$\|T_{Re(k),\mu}\|_{L^2(\dd,\mu)\to L^2(\dd,\mu)}\leq \|T_{k,\mu}\|_{L^2(\dd,\mu)\to L^2(\dd,\mu)}.$$
%
%
 Let $$f\in L^2(\dd,\mu)$$ be a real-valued function. Then
\begin{equation*}\langle T_{Re(k),\mu}f,f \rangle_{L^2(\dd,\mu)}=\langle T_{k,\mu}f, f \rangle_{L^2(\dd,\mu)}.
\end{equation*}
If $T_{Re(k),\mu}$ is bounded on $L^2(\dd,\mu)$, then there is a constant $c$ such that for any real-valued function $$f\in L^2(\dd,\mu),$$ we have
$$\|T_{k,\mu}f\|_{L^2(\dd,\mu)}\leq c\|f\|_{L^2(\dd,\mu)}.$$
By Lemma \ref{L:bounded on real enough}, $T_{k,\mu}$ is bounded on $L^2(\dd,\mu)$.
So, for the rest of the proof, it is sufficient to show that $$\|T_{k,\mu}\|_{L^2(\dd,\mu)\to L^2(\dd,\mu)}\le {2} \|T_{Re(k),\mu}\|_{L^2(\dd,\mu)\to L^2(\dd,\mu)}.$$

\noindent Since $T_{k,\mu}$ is bounded and positive, we have
 \begin{eqnarray}\label{E:norm positive T_K}
\|T_{k,\mu}\|_{L^2(\dd,\mu)\to L^2(\dd,\mu)}&=&\sup\big\{\| T_{k,\mu}^{\frac{1}{2}}f\|^2_{L^2(\dd, \mu)}:  \|f\|_{L^2(\dd, \mu)} \le 1 \big\}\nonumber\\
& \leq & 2 \sup\big\{\| T_{k,\mu}^{\frac{1}{2}}f\|^2_{L^2(\dd, \mu)}: f \ \text{is real-valued and  } \|f\|_{L^2(\dd, \mu)} \le 1 \big\}\nonumber\\
&= & {2} \sup\big\{\langle T_{k,\mu}f, f\rangle_{L^2(\dd, \mu)}:  f \ \text{is real-valued and  } \|f\|_{L^2(\dd, \mu)} \le 1 \big\}\nonumber\\
&=&2 \sup\big\{\langle T_{Re(k),\mu}f, f\rangle_{L^2(\dd, \mu)}: f \ \text{is real-valued and  } \|f\|_{L^2(\dd, \mu)} \le 1 \big\}\nonumber\\
&\leq &2 \|T_{Re(k),\mu}\|_{L^2(\dd,\mu)\to L^2(\dd,\mu)}.\nonumber
\end{eqnarray}
\end{proof}


\begin{lemma}\label{L: K reproducing, T_K T_RK}
If $k(z, w)$ is a reproducing kernel for some Hilbert space $H$ of functions  over $\dd$ such that for any $z\in\dd$ the function $k(z,\cdot)$ is continuous over $\overline{\dd}$ and $T_{k,\mu}$ is defined as in Lemma \ref{T:Carleson,real kernel,kernel}  on  $L^2(\dd, \mu)$,  then $T_{k,\mu}$ is positive.
\end{lemma}

\begin{proof} For $$f\in L^2(\dd, \mu),$$
we have
\begin{eqnarray*}
\langle T_{k,\mu}f, f \rangle_{L^2(\dd, \mu)}&=&\int_{\dd}\int_{\dd}k(z,w)f(w)d\mu(w)\overline{f(z)}d\mu(z)\\
&=&\int_{\dd}\int_{\dd}\Big\langle k(\cdot,w), k(\cdot, z) \Big\rangle_H f(w)d\mu(w)\overline{f(z)}d\mu(z)\\
&=&\Big\langle\int_{\dd}k(u,w)f(w)d\mu(w), \int_{\dd}k(u,z)f(z)d\mu(z)\Big\rangle_H\\
&=&\langle T_{k,\mu}f, T_{k,\mu}f\rangle_H\\
&\geq &0.
\end{eqnarray*}
\end{proof}

\noindent The rest of the proof of Lemma \ref{T:Carleson,real kernel,kernel} follows from \cite{ARS}. For completeness, we sketch a short,  slightly different proof. We overlook a needed density argument to simplify the presentation.

\bigno \textbf{Proof of Part $(b)\Rightarrow$ \textbf{Part} $ (c)$ in Lemma \ref{T:Carleson,real kernel,kernel}. } Let $I$ be the identity  map from $H\to L^2(\dd, \mu),$ i.e.,
$$
I(f)=f,\ \ \ f\in H.
$$
Since $H$ is spanned by functions of the form $k(\cdot, w)$, ($w\in\dd$), $I$ is densely defined.  Moreover, $I$ is clearly closed, so the adjoint $I^*$ is also a densely defined, closed operator.
Formally, $I^*$ is given by
$$
I^*f(z)=\int_{\dd}f(w)k(z,w)d\mu(w).
$$
Indeed, for $f\in$ Dom$(I)$ with $I^*f\in H,$ we have,
by the reproducing property,
\begin{eqnarray*}
I^*f(z)&=&\langle I^*f, k_z\rangle_{H}\\
&=&\langle f, k_z\rangle_{L^2(\dd, \mu)}\\
&=&\int_{\dd}f(w)k(z,w)d\mu(w).
\end{eqnarray*}
Then
\begin{eqnarray}\label{E: norm I and whole T}
\|I^*f\|^2_{H}&=&\langle I^*f, I^*f\rangle_{H}\nonumber\\
&=&\Big\langle\int_{\dd}f(w)k(z,w)d\mu(w),\ \int_{\dd}f(u)k(z,u)d\mu(w)\Big\rangle
_{H}\nonumber\\
&=&\int_{\dd}\int_{\dd}k(u,w)f(w)d\mu(w)\overline{f(u)}d\mu(u)\nonumber\\
&=&\langle T_{k,\mu}f, f\rangle_{L^2(\dd, \mu)}.
\end{eqnarray}
This implies $(b)\Rightarrow (c).$

\bigno \textbf{Proof of Part $(c)\Rightarrow$ \textbf{Part} $(a)$ in Lemma \ref{T:Carleson,real kernel,kernel}. }
It suffices to consider the restriction of $T_{Re(k),\mu}$ to real-valued functions in $L^2(\dd, \mu)$ by Lemma \ref{L:bounded on real enough}. Since $$k(z,w)=\overline{k(w,z)},$$  $T_{Re(k),\mu}$ is a symmetric  operator on $L^2({\dd, \mu}).$
Condition $(c)$ means that $I^*$ is bounded, it follows from (\ref{E: norm I and whole T}) that  $$\text{Dom}\big(T_{Re(k),\mu}\big)=L^2(\dd, \mu).$$
Therefore,
$T_{Re(k),\mu}$ is bounded. 

\end{proof}

\bigno Now we come to a simple yet pivotal point in the proof. The next lemma is easily proved by direct calculation, but its discovery is  a fortunate coincidence since it requires us to explicitly factor a given kernel into the convolution of two.  This is usually impossible. We can make it work only for the standard Dirichlet space, and generalizing it to other spaces may require different ideas.

\begin{lemma}\label{T:K Carleson = Dirichlet}
Suppose that $\mu$ is a  positive measure on $\dd$. It is a  $\mathcal{D}$-Carleson measure if and only if $K_1: L^2(\dd)\to L^2(\dd, \mu)$ is bounded, where
\begin{equation*}
K_1f(z)=\int_\dd\frac{f(w)}{1-z\bar{w}}dA(w).
\end{equation*}
\end{lemma}

\bigno Before the proof we observe a result to reduce $L^p(\dd),$ the domain of $K_1,$ to $L_a^p(\dd)$ via the Bergman projection $P$ for any $p>1$.  The proof  is skipped.

\begin{lemma}\label{L:K=KP Lp to X}
If $1<p<\infty$, then  $K_1f=K_1(Pf)$
for all $f\in L^p(\dd).$
\end{lemma}

\begin{proof}[Proof of Lemma \ref{T:K Carleson = Dirichlet}]
Let $\widetilde{K}_1$ be the restriction of  $K_1$ onto $L_a^2(\dd)$. By
Lemma \ref{L:K=KP Lp to X}, $$K_1:L^2(\dd)\to L^2(\dd,\mu)$$ is bounded if and only if $$\widetilde{K}_1:L_a^2(\dd)\to L^2(\dd,\mu)$$ is bounded if and only if $$\widetilde{K}_1\widetilde{K}^*_1:L^2(\dd,\mu)\to L^2(\dd,\mu)$$ is bounded. With direct calculations,
\begin{equation*}
\widetilde{K}_1\widetilde{K}^*_1f(z)=\int_{\dd}f(w)k_{\mathcal{D}}(z,w)d\mu(w),
\end{equation*}
where
$$
k_{\mathcal{D}}(z,w)=\frac{1}{z\bar{w}}\log
\frac{1}{1-z\bar{w}}.
$$
 So Lemma  \ref{T:Carleson,real kernel,kernel} completes the proof.


\end{proof}

\bigno Now the hard work begins. Our target is  $$K_1:L^2(\dd)\to L^2(\dd,\sigma),$$ which  is a special case of the notorious two weight problem, whose full resolution seems out of reach.  But we manage to get what is sufficient to resolve our problem. Indeed our arguments can prove much more, although we  state our result only for the sleekest case (Theorem \ref{T:main0405}). On the other hand, although the techniques we use are not new in harmonic analysis,  the way they are modified (from \cite{Tolsa2014}) and applied to operator theory in this paper should be applicable to some other problems in operator theory.

\begin{definition}\label{D:R} \emph{\cite{FW2014}} A measure $\sigma$ on the unit disk has the  reverse doubling property if there is a constant $\delta<1$ such that
\begin{equation*}\label{E:reverseconstant}
\frac{|B_I|_\sigma}{|Q_I|_\sigma}<\delta,
\end{equation*}
for any interval $I\subset\tt$. Here $$Q_I=\{z\in\dd:1-|I|\leq |z|<1,\frac{z}{|z|}\in I\}$$ and
$$B_I=\{z\in\dd:1-\frac{|I|}{2}< |z|<1,\frac{z}{|z|}\in I\}.$$
\end{definition}

\bigno Starting now we present  several lemmas with  general parameters $(p, q)$ instead of mere $(2, 2)$. There are two reasons to do this. First of all, this should appeal to those interested in the (rather large) literature of generalized Dirichlet spaces and it is indeed at almost no extra cost for us.   Second, some---including us---might find   $(p, q)$ arguments   to be more illustrative and they won't really complicate the reading of any serious reader.

\begin{lemma}\label{L:twp reverse doubling}Let $\alpha>0$ and $1<p\leq q<\infty$.
Let $\mu$ and $\nu$ be weights on $\dd$ such that $\mu^{1-p'}$ and
$\nu$ have the reverse doubling property. If
\begin{equation*}\label{E:rpq}
\sup_{I\subset\tt}\frac{|Q_I|_\nu^\frac{1}{q}|Q_I|_{\mu^{1-p'}}
^{\frac{1}{p'}}}{|Q_I|^{\frac{\alpha}{2}}}<\infty,
\end{equation*}
then there exists a constant $c$ such that for  $f \in L^p(\dd, \mu)$,
\begin{equation*}\label{E:kpq}
\Big(\int_\dd |K_\alpha (f)|^q\nu(z)dA(z)\Big)^{\frac{1}{q}}
\leq c\Big(\int_\dd |f(z)|^p\mu(z)dA(z)\Big)^\frac{1}{p}.
\end{equation*}
Here $p'$ is the conjugate index, i.e., $\frac{1}{p}+\frac{1}{p'}=1$, and
\begin{equation*}
K_\alpha f(z)=\int_\dd\frac{f(w)}{(1-z\bar{w})^\alpha}dA(w).
\end{equation*}
\end{lemma}


\bigno Let $$\zz_+=\nn \cup \{0\}.$$ Consider the following two dyadic grids on $\tt$,
\begin{equation*}
\mathcal{D}^0=\bigg\{\Big[\frac{2\pi m}{2^j},\frac{2\pi (m+1)}{2^j}\Big):m\in\mathbb{Z}_+, j\in\mathbb{Z}_+, 0\leq m<2^j\bigg\}
\end{equation*}
and
\begin{equation*}
\mathcal{D}^{\frac{1}{3}}=\bigg\{\Big[\frac{2\pi m}{2^j}+\frac{2\pi}{3},\frac{2\pi (m+1)}{2^j}+\frac{2\pi}{3}\Big):m\in\mathbb{Z}_+, j\in\mathbb{Z}_+, 0\leq m<2^j\bigg\}.
\end{equation*}
For each $\beta\in\{0, \frac{1}{3}\}$, let $\mathcal{Q}^\beta$ denote the collection of Carleson boxes $Q_I$ with $I\in\mathcal{D}^\beta$ and we call $\mathcal{Q}^\beta$  a Carleson box system over $\dd$.

\bigno The first appearance of shifted dyadic grids in print is probably in page 30 of \cite{Christ}. A quick way to appreciate why  shifted dyadic grids are powerful is to look at \cite{APR2013}, \cite{GJ1982} and \cite{Mei2003}. In particular, \cite{APR2013} contains a nice application to Sarason's problem on Toeplitz products.

\begin{lemma}\label{L:dyadic}\emph{\cite{Mei2003}}
Let $J\subset\tt$ be an interval. Then there exists an interval $$L \in\mathcal{D}^0\cup\mathcal{D}^\frac{1}{3}$$ such that
$$J\subset L \text{ and } |L|\leq 6|J|.$$
\end{lemma}

\bigno The proof of the next   lemma will be  skipped.

\begin{lemma}\label{L:Key01} There is a positive constant $c$ such that for any $z, w \in \dd$, there exists  a Carleson box $Q_I$ such that $z,w\in Q_I$ and
\begin{equation*}\frac{1}{c}|Q_I|^{\frac{1}{2}}\leq |1-z\bar{w}|\leq c|Q_I|^{\frac{1}{2}}.
\end{equation*}
\end{lemma}

\bigno By Lemma \ref{L:dyadic} and Lemma \ref{L:Key01}, there is a constant $c$ such that for any $z,w\in\dd$, we can find an $$L \in\mathcal{D}^0\cup\mathcal{D}^\frac{1}{3}$$ such that
\begin{equation}\label{E:kernelestimate}
\frac{1}{|1-z\bar{w}|^\alpha}\leq
c\frac{\chi_{Q_L}(z)\chi_{Q_L}(w)}{|Q_{L}|^\frac{\alpha}{2}}.
\end{equation}
\bigno For each $\beta\in\{0,\frac{1}{3}\}$, we define
\begin{equation*}
K^\beta_\alpha f(z)=\sum_{I\in\mathcal{D}^\beta}
\int_{\dd}\frac{f(w)\chi_{Q_I}(w)}{|Q_I|^\frac{\alpha}{2}}dA(w)\chi_{Q_I}(z).
\end{equation*}

\noindent By (\ref{E:kernelestimate}), for any positive function $f$ on $\dd$,  we have
\begin{equation*}
|K_\alpha f(z)|\leq c(K^0_\alpha f(z)+K^\frac{1}{3}_\alpha f(z))
\end{equation*}
for any $z \in \dd$. It easily follows that
\begin{lemma}\label{L:Form K to Dyadic}
Let $\sigma$ and $\omega$ be two weights on $\dd$ and $1\leq p, q\leq\infty$. If both
$K^0_\alpha$ and $K^{\frac{1}{3}}_\alpha$ are bounded from $L^p(\dd,\sigma)$ into $L^q(\dd,\omega)$,
then $$K_\alpha:L^p(\dd,\sigma)\to L^q(\dd,\omega)$$ is  bounded.
\end{lemma}

\bigno Now we come to another key technical point. Namely, we prove an off-diagonal Carleson embedding theorem  over the unit disk.  A good entry point to this area of techniques is \cite{Tolsa2014} which contains various Carleson embedding theorems of diagonal type over the Euclidean space $\rr^n$.
For operator theorists, it is probably desirable in general to see how to adopt those (rich) techniques on $\rr^n$ to the study of analytic function spaces on the unit disk $\dd$.

\bigno
\noindent Let $\sigma$ be a weight on $\dd$ and $\beta\in\{0,\frac{1}{3}\}$.
Let $\mathcal{Q}^{\beta}$ be a dyadic Carleson system over $\dd$. For any $$Q_I\in\mathcal{Q}^\beta,$$
let \begin{equation*}
\ee_{Q_I}^\sigma f=\frac{1}{|Q_I|_\sigma}\int_{Q_I}f(z)\sigma(z)dA(z).
\end{equation*}
Then a tree mapping on the dyadic Carleson system $\mathcal{Q}^\beta$ is given by
\begin{equation*}
\Lambda_\beta: f\mapsto \ee_{Q_I}^\sigma f
\end{equation*}
for $f\in L^1(\dd,\sigma)$.

\bigno Next, we endow $$\mathcal{Q}^\beta, \beta\in\{0,\frac{1}{3}\},$$ with a measurable space structure. 
For any $t\in \rr$ and
$Q_I\in\mathcal{Q}^\beta$, let $$a_t(Q_I)=|Q_I|_\sigma^{t}.$$ Then $\{\mathcal{Q}^\beta, a_t(\cdot)\}$
is a measurable space.
Furthermore, for $1\leq p<\infty$, $$f\in L^p(\mathcal{Q}^\beta,a_t(\cdot))$$
means
\begin{equation*}
\norm{f}_{L^p(\mathcal{Q}^\beta,a_t(\cdot))}
=\Big(\sum_{Q_I\in\mathcal{Q}^\beta}a_t(Q_I)|f(Q_I)|^p\Big)^\frac{1}{p}<\infty.
\end{equation*}
Moreover, $f \in L^{p,\infty}(\mathcal{Q}^\beta,a_t(\cdot))$  {means}
\begin{equation*}
\norm{f}_{L^{p,\infty}(\mathcal{Q}^\beta,a_t(\cdot))}=
\sup_{\lambda>0}\bigg\{\lambda\bigg[\sum_{Q_I\in \mathcal{Q}_\lambda}a_t(Q_I)\bigg]^{\frac{1}{p}}\bigg\}<\infty,
\end{equation*}
where $$\mathcal{Q}_\lambda=\{Q_I\in\mathcal{Q}^\beta:
|f(Q_I)|>\lambda\}.$$


\begin{lemma}\label{T:pqCarlesonembedding}Let $1<p\leq q<\infty$, $t=\frac{q}{p}$ and $\sigma$ be a weight on $\dd$. Let $\mathcal{Q}^{\beta}$ be a dyadic Carleson system on $\dd$, $\beta \in \{0,\frac{1}{3}\}$. Let $$\Lambda_\beta: f\mapsto \ee_{Q_I}^\sigma f$$
be the tree mapping on $\mathcal{Q}^\beta$. If there is a constant $c_1$ such that
for any $K \in \mathcal{D}^\beta$,
\begin{equation*}
\sum_{Q_I\in\mathcal{Q}^\beta:Q_I\subset Q_K}|Q_I|_\sigma^t\leq c_1
|Q_K|_\sigma^t,
\end{equation*}
then the tree mapping $\Lambda_\beta$ is bounded from $L^p(\dd,\sigma)$ into $L^q(\mathcal{Q}^\beta, a_t(\cdot))$. That is, there is a constant $c_2$ such that
\begin{equation*}
\Big(\sum_{Q_I\in\mathcal{Q}^\beta}|Q_I|_\sigma^t(\ee^\sigma_{Q_I}f)^q\Big)^{\frac{1}{q}}\leq c_2
\Big(\int_\dd |f(z)|^p\sigma(z)dA(z)\Big)^\frac{1}{p}, \quad f \in L^p(\dd, \sigma).
\end{equation*}
\end{lemma}

\bigno In order to prove Lemma \ref{T:pqCarlesonembedding}, we need the Marcinkiewick interpolation theorem which we recall for the convenience of the readers.
\begin{lemma}\emph{(\cite{BL1976}, Page 6)}\label{L:marcinliewick} Let $0<p_0,q_0,p_1,q_1\leq\infty$ and $p_0\not=p_1$.
Let $(X,\sigma)$ and $(Y,\omega)$ be measurable spaces. Let $T$ be a linear operator such that
\begin{itemize}
\item[(a)] $T:L^{p_0}(X,\sigma)\to L^{q_0,\infty}(Y,\omega)$ is bounded with norm $c_3$;
\item[(b)] $T:L^{p_1}(X,\sigma)\to L^{q_1,\infty}(Y,\omega)$ is bounded with norm $c_4$.
\end{itemize}
Set $$\frac{1}{p}=\frac{1-\theta}{p_0}+\frac{\theta}{p_1}$$ and $$\frac{1}{q}=\frac{1-\theta}{q_0}+\frac{\theta}{q_1}$$ for some $0\leq \theta\leq 1$.
If $p\leq q$, then
\begin{equation*}
T:L^p(X,\sigma)\to L^q(Y,\omega)
\end{equation*}
is bounded with norm $c_5=c(c_3, c_4, \theta).$
\end{lemma}

\begin{proof}[Proof of Lemma \ref{T:pqCarlesonembedding}]Let $$\beta\in \{0,\frac{1}{3}\}$$ be fixed. First,  the tree mapping
\begin{equation*}
\Lambda_\beta: f\to \ee_{Q_I}^\sigma f=\frac{1}{|Q_I|_\sigma}\int_{Q_I}f(z)\sigma(z)dA(z)
\end{equation*}
is bounded from $L^\infty(\dd,\sigma)$ into $L^\infty(\mathcal{Q}^\beta, a_t(\cdot))$ with norm $c_3 \leq 1$.

\bignobf{Claim:}  $\Lambda_\beta: L^1(\dd,\sigma)\to L^{t,\infty}(\mathcal{Q}^\beta, a_t(\cdot))$ is bounded.

\bigno If we can verify this claim, then let $$\theta=\frac{1}{p},$$ and it follows that $$\Lambda_\beta: L^p(\dd,\sigma)\to L^q(\mathcal{Q}^\beta, a_t(\cdot))$$ is bounded which will complete the proof of Lemma \ref{T:pqCarlesonembedding}.

\bigno
\emph{Proof of Claim:} Let $f>0$ and $\lambda>0$. Let $$\mathcal{Q}_\lambda=\{Q_I\in\mathcal{Q}^\beta:
|\ee_{Q_I}^\sigma f|>\lambda\}.$$
Let $$\{Q_{I_j}:j\in \Gamma\}$$ be the collection of maximal (with respect to inclusion) Carleson boxes in $\mathcal{Q}_\lambda$ for some index set $\Gamma$.  Then
 \begin{eqnarray*}
 \sum_{Q_I\in \mathcal{Q}_\lambda}a_t(Q_I) &=&\sum_{j\in\Gamma} \sum_{Q_I\subset
 Q_{I_j}} |Q_I|_\sigma^t\\
 &\leq& c_1\sum_{j\in\Gamma} |Q_{I_j}|_\sigma^t\\
 &\leq& c_1[\sum_{j\in\Gamma} |Q_{I_j}|_\sigma]^t.
 \end{eqnarray*}
Now the proof is complete by observing
 \begin{eqnarray*}
 \sum_{j\in \Gamma}|Q_{I_j}|_\sigma&\leq& \frac{1}{\lambda}\sum_{j\in\Gamma}\int_{Q_{I_j}}f(z)\sigma(z)dA(z)\\
 &\leq&
 \frac{1}{\lambda}\int_\dd f(z)\sigma(z)dA(z).
 \end{eqnarray*}
\end{proof}

\bigno Next, we need another result on the reverse doubling property {from \cite{FW2014} (Lemma 7)}.

\begin{lemma}\label{L:carlesonc} Let $\beta\in\{0, \frac{1}{3}\}$. Let $\mathcal{D}^\beta$ be a dyadic grid of $\tt$, and $\mathcal{Q}^{\beta}$ be a dyadic Carleson system over $\dd$. If a weight $\sigma$ has the reverse doubling property, then there is a constant $c$ such that for any $K \in \mathcal{D}^\beta$
\begin{equation}\label{E:carlesoncondition}
\sum_{Q_I\in\mathcal{Q}^\beta:Q_I\subset Q_K}|Q_I|_\sigma\leq c|Q_K|_\sigma.
\end{equation}
\end{lemma}

\bigno (\ref{E:carlesoncondition}) is known as the Carleson embedding condition.
 Combining with the result in \cite{ARS2002}, one may conjecture that this condition a necessary and sufficient condition for Carleson measures on $\mathcal{D}$.

\bigno Combining Lemma \ref{T:pqCarlesonembedding} and Lemma \ref{L:carlesonc}, we have
\begin{corollary}\label{C:carleson03} Let $1<p\leq q<\infty, t=\frac{q}{p}$ and $\sigma$ be a weight with the reverse doubling property. Then there is a constant $c$ such that
\begin{equation*}
\Big(\sum_{Q_I\in\mathcal{Q}^\beta}|Q_I|_\sigma^t(\ee^\sigma_{Q_I}f)^q\Big)^{\frac{1}{q}}\leq c
\Big(\int_\dd |f(z)|^p\sigma(z)dA(z)\Big)^\frac{1}{p}.
\end{equation*}
\end{corollary}

\bigno Next, we   prove an inequality for $K_\alpha^\beta$  whose formulation is of independent interests, although the proof is more or less standard now.

\begin{lemma}\label{L:twp reverse doubling d}Let $1<p\leq q<\infty$.
Let $\mu$ and $\nu$ be weights on $\dd$ such that $\mu^{1-p'}$ and
$\nu$ have the reverse doubling property. If
\begin{equation*}\label{E:rpq}
\sup_{I\subset\tt}\frac{|Q_I|_\nu^\frac{1}{q}|Q_I|_{\mu^{1-p'}}
^{\frac{1}{p'}}}{|Q_I|^{\frac{\alpha}{2}}}<\infty,
\end{equation*}
then there exists a constant $c_1$ such that for  $f \in L^p(\dd, \mu)$,
\begin{equation*}\label{E:kpq}
\Big(\int_\dd |K^{\beta}_\alpha(f)|^q\nu(z)dA(z)\Big)^{\frac{1}{q}}
\leq c_1\Big(\int_\dd |f(z)|^p\mu(z)dA(z)\Big)^\frac{1}{p}.
\end{equation*}
\end{lemma}

\begin{proof}
It is sufficient to show that there is a constant $c_1$ such that
\begin{equation}\label{E:norm inequality two two}
\Big(\int_\dd |K^\beta_\alpha(f\mu^{1-p'})|^q\nu(z)dA(z)\Big)^{\frac{1}{q}}\leq
c_1\Big(\int_\dd |f(z)|^p \mu^{1-p'}(z) dA(z)\Big)^\frac{1}{p}.
\end{equation}
Then  Lemma \ref{E:kpq} holds by replacing $f$ in (\ref{E:norm inequality two two}) with $f\mu^{p'-1}$.

\bigno Let $$f\in L^p(\dd, \mu^{1-p'})$$ and $$g\in L^{q'}(\dd,\nu).$$ Let $$c_2=\sup_{I\subset\tt}\frac{|Q_I|_\nu^\frac{1}{q}|Q_I|_{\mu^{1-p'}}
^{\frac{1}{p'}}}{|Q_I|^{\frac{\alpha}{2}}}.$$ Then
\begin{eqnarray*}
|\langle K^\beta_\alpha (f\mu^{1-p'}), g\rangle_{L^2(\nu)}|&=&\bigg|\sum_{Q_I\in\mathcal{Q}^\beta}
\frac{1}{|Q_I|^\frac{\alpha}{2}}\int_{Q_I}f(z)\mu^{1-p'}(z)dA(z)\int_{Q_I}g(z)
\nu(z)dA(z)\bigg|\\
&\leq& c_2\sum_{Q_I\in\mathcal{Q}^\beta}\frac{1}{|Q_I|_\nu^{\frac{1}{q}}}
\frac{1}{|Q_I|_{\mu^{1-p'}}^\frac{1}{p'}}
\int_{Q_I}|f(z)|
\mu^{1-p'}(z)dA(z)\\&\times&\int_{Q_I}|g(z)|\nu(z)dA(z)\\
&=&c_2\sum_{Q_I\in\mathcal{Q}^\beta} |Q_I|_{\mu^{1-p'}}^\frac{1}{p}
\frac{1}{|Q_I|_{\mu^{1-p'}}}\int_{Q_I}|f(z)|
\mu^{1-p'}(z)dA(z)\\
&& \times \
|Q_I|_\nu^{\frac{1}{q'}}\frac{1}{|Q_I|_\nu}\int_{Q_I}|g(z)|\nu(z)dA(z)\\
&\leq&c_2\bigg(\sum_{Q_I\in\mathcal{Q}^\beta}|Q_I|_{\mu^{1-p'}}^{\frac{q}{p}}
\Big(\frac{1}{|Q_I|_{\mu^{1-p'}}}\int_{Q_I}|f(z)|\mu^{1-p'}(z)dA(z)\Big)^q\bigg)^{\frac{1}{q}}\\
& & \times \  \bigg(\sum_{Q_I\in\mathcal{Q}^\beta}|Q_I|_{\nu}
\Big(\frac{1}{|Q_I|_\nu}\int_{Q_I}|g(z)|\nu(z)dA(z)\Big)^{q'}\bigg)^{\frac{1}{q'}}.
\end{eqnarray*}

\bigno Since $\mu^{1-p'}$ and $\nu$ have the reverse doubling property, by the $p$-$q$ Carleson embedding (Corollary \ref{C:carleson03}),
\begin{equation*}
\bigg(\sum_{Q_I\in\mathcal{Q}^\beta}|Q_I|_{\mu^{1-p'}}^{\frac{q}{p}}
\Big(\frac{1}{|Q_I|_{\mu^{1-p'}}}\int_{Q_I}|f(z)|\mu^{1-p'}(z)dA(z)\Big)^q\bigg)^{\frac{1}{q}}
\leq c_3\Big(\int_\dd |f(z)|^p\mu^{1-p'}(z)dA(z)\Big)^\frac{1}{p}
\end{equation*}
and
\begin{equation*}
\bigg(\sum_{Q_I\in\mathcal{Q}^\beta}|Q_I|_{\nu}
\Big(\frac{1}{|Q_I|_\nu}\int_{Q_I}|g(z)|\nu(z)dA(z)\Big)^{q'}\bigg)^{\frac{1}{q'}}
\leq c_4\Big(\int_\dd |g(z)|^{q'}\nu(z)dA(z)\Big)^\frac{1}{q'}.
\end{equation*}

\bigno Let $$c_1=c_2c_3c_4.$$ Then
we have
\begin{equation*}
|\langle K^\beta_\alpha f\mu^{1-p'}, g\rangle_{L^2(\nu)}|\leq c_1 \Big(\int_\dd |f(z)|^p\mu^{1-p'}(z)dA(z)\Big)^\frac{1}{p}
\Big(\int_\dd |g(z)|^{q'}\nu(z)dA(z)\Big)^\frac{1}{q'}.
\end{equation*}
The proof of Lemma \ref{L:twp reverse doubling d} is complete now.
\end{proof}

\bigno Now Lemma \ref{L:twp reverse doubling} follows from Lemma \ref{L:Form K to Dyadic} and Lemma  \ref{L:twp reverse doubling d}.  Then  Theorem \ref{T:main0405}  follows from Lemma \ref{L:twp reverse doubling} when applied to $K_1$ with  $p=q=2$ and $\mu$ being the Lebesgue measure.

\bignobf{Acknowledgement}

\medno  G. Cheng is supported by NSFC (11471249),
 Zhejiang Provincial Natural Science Foundation of China (LY14A010021).
  X. Fang  is supported by NSC of Taiwan (106-2115-M-008-001-MY2) and NSFC 11571248 during his visit to Soochow University in China.
Z. Wang is supported by NSFC (11601296) and NSF of Shaanxi (2017JQ1008).
 J. Yu is supported by NSFC (11501384).

\bigskip

\bigskip

G. Cheng, Department of Mathematics, Sun Yat-sen University, Guangzhou 510275, P. R. China
E-mail address: chgzh09@gmail.com

\bigno

X. Fang, Department of Mathematics, National Central University, Chung-Li
32001, Taiwan; Email: xfang@math.ncu.edu.tw

\bigno

Z. Wang, College of Mathematics and Information Sciences, Shaanxi Normal University, Xi'an 710062,
P. R. China; Email: zipengwang@snnu.edu.cn

\bigno

J. Yu, School of Mathematics, Sichuan University, Chengdu 610064, P. R. China; Email: jiayangyu@scu.edu.cn


\begin{thebibliography}{99}








%
%
%
\bibitem{APR2013}
A. Aleman, S. Pott and M. Reguera, Sarason conjecture on the Bergman space, Int. Math. Res. Notices. (IMRN), \textbf{2017}(2017), no. 14, 4320-4349.

\bibitem{ARS2002} N. Arcozzi, R. Rochberg and E. Sawyer, Carleson measures for analytic Besov spaces, Rev. Mat. Iberoamericana \textbf{18} (2002), no. 2, 443-510.
\bibitem{ARS}
N. Arcozzi, R. Rochberg and E. Sawyer, Carleson measures for the Drury-Arveson Hardy space
and other Besov-Sobolev spaces on complex balls,
 Adv. Math. {\bf{218}} (2008),  1107-1180.




\bibitem{BL1976}
J. Bergh and J. Lofstrom, Interpolation Spaces-An Introduction,
Grundlehren der Mathematischen Wissenschaften,  \textbf{223}, Springer-Verlag, 1976.

\bibitem{Christ}
M. Christ, Weak Type $(1, 1)$ Bounds for Rough Operators, Ann. Math. \textbf{128} (1988) , 19-42.













\bibitem{Dirichlet book}
O. El-Fallah, K. Kellay, J. Mashreghi and T. Ransford,  A Primer on the Dirichlet Space, Cambridge Tracts in Mathematics, {\bf{203}}, Cambridge University Press, Cambridge, 2014.
\bibitem{EKMR2014}
O. El-Fallah, K. Kellay, J. Mashreghi and T. Ransford, One-box conditions for Carleson measures for the Dirichlet space, Proc. Amer. Math. Soc. \textbf{143} (2015), no. 2, 679-684.



\bibitem{GJ1982}
J. Garnett and P. Jones, BMO from dyadic BMO, Pacific J. Math. \textbf{2}(1982), 351-371.

\bibitem{FW2014}
X. Fang and Z. Wang, Two weight inequalities for the Bergman projection with doubling measures, Taiwanese J.  Math. \textbf{19} (2015), no. 3, 919-926.












\bibitem{Mei2003}
T. Mei, BMO is the intersection of two translates of dyadic BMO, C. R. Acad. Sci. Paris, Ser. I. \textbf{336} (2003), 1003-1006.










\bibitem{Ste1980}
D. Stegenga, Multipliers on the Dirichlet space, Illinois J. Math. \textbf{24} (1980), no.1, 113-139.

\bibitem{stein}
E. Stein, Harmonic Analysis: Real-Variable Methods, Orthogonality, and Oscillatory Integrals, Princeton Math. Ser., vol. 43, Princeton University Press, Princeton, NJ, 1993.

\bibitem{Tolsa2014}
X. Tolsa, Analytic Capacity, the Cauchy Transform, and Non-homogeneous Calderon-Zygmund Theory, Progress in Mathematics, \textbf{307}, Birkhauser/Springer, 2014.






\bibitem{Wu1999}
Z. Wu, Carleson measures and multipliers for Dirichlet spaces, J. Funct. Anal. {\bf{169}} (1999), no. 1, 148-163.
\bibitem{Wu2011}
Z. Wu, A new characterization for Carleson measures and some applications, Integr.
Equ. Oper. Theory \textbf{71} (2011), 161-180.

\bibitem{Zhu}
K. Zhu, Operator Theory in Function Spaces, Second edition, Mathematical Surveys and Monographs, {\bf{138}},
American Mathematical Society, Providence, RI, 2007.



\end{thebibliography}
\end{document}